\RequirePackage{fix-cm}
\documentclass[smallextended]{svjour3}       
\smartqed  
\usepackage{}
\usepackage{graphicx}
\usepackage{multirow,bigstrut}
\usepackage{amsmath,amsfonts,amssymb}
\usepackage{mathrsfs}
\usepackage{color}
\usepackage{indentfirst}
\usepackage{algorithm,algpseudocode}
\usepackage{latexsym}
\usepackage{graphicx}
\usepackage{array,epsf,epsfig}
\usepackage{mathdots}
\usepackage{cite}
\usepackage{multirow}
\usepackage{epstopdf}
\usepackage{mathtools}
\usepackage{tensor}
\usepackage{textcomp}
\usepackage{booktabs}

\newcommand{\mO}{{\mathcal{O}}}

\newcommand{\tT}{\intercal}

\newcommand{\beq}{\begin{equation}}
\newcommand{\eeq}{\end{equation}}
\newcommand{\bey}{\begin{eqnarray}}
\newcommand{\eey}{\end{eqnarray}}

\begin{document}

\title{Exponential-sum-approximation technique for variable-order time-fractional diffusion equations \thanks{{\bf{Funding:}} This work is supported in part by research grants of the Science and Technology Development Fund, Macau SAR (file no. 0118/2018/A3), and MYRG2018-00015-FST from University of Macau.}}


\author{Jia-Li Zhang \and Zhi-Wei Fang \and Hai-Wei Sun
}


\institute{Jia-Li Zhang \at
              Department of Mathematics, University of Macau,
Macao\\
\email{zhangjl2628@gmail.com}\\
           Zhi-Wei Fang \at
              Department of Mathematics, University of Macau,
Macao\\
\email{fzw913@yeah.net}\\
           Hai-Wei Sun(Corresponding author) \at
              Department of Mathematics, University of Macau,
Macao\\
              \email{HSun@um.edu.mo}
}

\date{Received: date / Accepted: date}

\maketitle

\begin{abstract}
In this paper, we study the variable-order (VO) time-fractional diffusion equations. For a VO function $\alpha(t)\in(0,1)$, we develop an exponential-sum-approximation (ESA) technique to approach the VO Caputo fractional derivative. The ESA technique keeps both the quadrature exponents and the number of exponentials in the summation unchanged at different time level. Approximating parameters are properly selected to achieve the efficient accuracy. Compared with the general direct method, the proposed method reduces the storage requirement from $\mO(n)$ to $\mO(\log^2 n)$ and the computational cost from $\mO(n^2)$ to $\mathcal{O}(n\log^2 n)$, respectively, with $n$ being the number of the time levels. When this fast algorithm is exploited to construct a fast ESA scheme for the VO time-fractional diffusion equations, the computational complexity of the proposed scheme is only of $\mO(mn\log^2 n)$ with $\mO(m\log^2n)$ storage requirement, where $m$ denotes the number of spatial grids. Theoretically, the unconditional stability and error analysis of the fast ESA scheme are given. The effectiveness of the proposed algorithm is verified by numerical examples.
\end{abstract}

\section{Introduction}
In this paper, we consider the numerical discretization of the \emph{variable-order}(VO) Caputo time-fractional diffusion equations
\begin{align}
&\prescript{C}{0}{\mathcal{D}}^{\alpha(t)}_tu(x,t)=\frac{\partial^2u(x,t)}{\partial x^2}+f(x,t),\ \ x\in (0,x_R),\ \ t\in (0,T], \label{VOTFDE1}\\
&u(x,0)=\varphi(x),\ \ x\in [0,x_R],\label{VOTFDE2}\\
&u(0,t)=u(x_R,t)=0,\ \ t\in (0,T],\label{VOTFDE3}
\end{align}
where $f(x,t)$ and $\varphi(x)$ are all smooth functions. The VO Caputo fractional derivative is defined by
\cite{Coimbra-2003}
\begin{align}\label{caputo-def}
\prescript{C}{0}{\mathcal{D}}^{\alpha(t)}_tu(t)\equiv\frac 1 {\Gamma(1-\alpha(t))}\int_0^t\frac{u'(\tau)}{(t-\tau)^{\alpha(t)}}d\tau,
\end{align}
where $\alpha(t)\in[\underline{\alpha},\overline{\alpha}]\subset(0,1)$ is the VO function and $\Gamma(\cdot)$ is the Gamma function.

In the last few decades, the fractional calculus played a more and more important role due to its broad applications in both mathematics and physical science. For example, it has been proved that the fractional calculus can better characterize certain complex phenomena in fields such as the biology, the ecology, the diffusion, and the control system \cite{Mainardi-2000, Raberto-2002, Benson-2000, Liu-2004, Podlubny-1999, Kilbas-2006}. Furthermore, it has been revealed that many important dynamical problems exhibit the order of the fractional operator varying with time, space, or some other variables; see \cite{Lorenzo-2002, Sun-2011, Sun-2019}. Therefore, the VO fractional operators become more and more useful in studying their memory properties associated with time or spatial location. Recently, the related applications of the VO fractional derivatives in the field of science and engineering have been reported; see, for details,
\cite{ Zhuang-2000, Pedro-2008, Coimbra-2003, Diazand-2009, Jia-2017, Sokolov-2005, Ingman-2004, Obembe-2017, Sun-2009, C}.

An interesting extension of the classical fractional calculus was proposed in 1993 by Samko and Ross \cite{Samko-1993} where they generalized the Riemann-Liouville and Marchaud fractional operators with the  VO sense. In 2002, the concept of the VO fractional derivative was firstly introduced by Lorenzo and Hartley \cite{Lorenzo-2002}. According to their definition, the order of the fractional derivative is allowed to vary as a function of independent variables such as time and space. Afterwards, various VO fractional differential operators with specific meanings were defined. Coimbra \cite{Coimbra-2003} gave a novel definition for the VO fractional differential operator by taking the Laplace-transform of the Caputo's definition of the fractional derivative. Soon, Coimbra, and Kobayashi \cite{Soon-2005} showed that Coimbra's definition was better suited for modeling physical problems due to its meaningful physical interpretations; see also \cite{Ramirez-2010}. 
Recently, several studies towards the approximation to the VO Caputo fractional derivative and the corresponding numerical analysis have been carried out. Shen et al. \cite{Shen-2012} employed a new numerical method to approximate the VO Caputo fractional derivative in the VO time-fractional diffusion equations. The convergence, stability, and solvability of the scheme were explored by the Fourier analysis. Zhao, Sun, and Karniadakis \cite{Zhao-2015} derived two second-order approximations to the VO Caputo fractional derivative and provided the error analysis. Du, Alikhanov, and Sun \cite{Du-2020} found a special point on each time interval to present a numerical differential formula for the VO Caputo fractional derivative, yielding at least second-order accuracy of the approximation. Furthermore, such an approximation was applied to devise two schemes for the multi-dimensional sub-diffusion equations with the VO time-fractional derivative. The stability and convergence of both schemes were also investigated via the energy method.

Due to the nonlocality of the fractional operators, however, the existing numerical methods for approximating the VO fractional derivative at each time level require all values in previous time levels. Therefore, in practical calculations, they are too expensive in storage and complexity. Even the constant-order (CO) fractional operators also suffer from such difficulty. Indeed, using the $L1$ approximation formula \cite{Oldham-1974, Sun-2006, Langlands-2005, Sun-2005, Liao-2018} to discretize the CO Caputo fractional derivative, it usually needs $\mathcal{O}(n)$ storage and $\mathcal{O}(n^2)$ computational cost, where $n$ is the total number of the time levels. To overcome this difficulty, many efforts have been made to speed up the evaluation of the CO Caputo fractional derivative. In fact there have been two ways to develop fast algorithms for the CO fractional derivatives and time-fractional diffusion equations. The first way is based on the fact that the coefficient matrices of the discretized CO fractional operators possess the lower triangular Toeplitz or Toeplitz-like structure. In \cite{Lu-2015, Lu-2018}, the approximate inversion method was proposed to solve the CO time-fractional diffusion equations with $\mathcal{O}(mn\log n)$ computational complexity, where $m$ denotes the number of spatial grid points. A block divide-and-conquer method \cite{Ke-2015} was presented to successfully solve the large-scale linear system arising in the CO time-fractional diffusion equations with $\mathcal{O}(mn\log^2 n)$ computational workload. Some other researchers further developed the Toeplitz-like matrix splitting based on all-at-once preconditioned fast method to solve space-time fractional diffusion equations \cite{Fu-2017, Bertaccini-2019}. Nevertheless, the coefficient matrix of the numerical schemes for the VO time-fractional diffusion equations will be no longer owning the Toeplitz-like structure. Therefore, those fast methods cannot be straightforwardly applied to the VO cases. The other way is based on the feature that the CO Caputo fractional derivative is the convolution operator. By making use of this property, a class of fast algorithms have been proposed via splitting the weak singular kernel in the convolution; see \cite{Lubich-2002, Schadle-2006, Lopez-2008}. As developments,
Baffet and Hesthaven\cite{Baffet-2016} compressed the kernel in the Laplace domain and obtained a sum-of-poles approximation for the Laplace transform of the kernel. Also using the Laplace transform technique, Jiang et al. \cite{Jiang-2017} adopted the Gauss-Jacobi quadrature and Gauss-Legendre quadrature to achieve a sum-of-exponentials for approximating the kernel function in the CO Caputo fractional derivative. The resulting algorithms by those methods require $\mathcal{O}(\log^2 n)$ storage and $\mathcal{O}(n\log^2 n)$ computational complexity. As we have mentioned above, however, the VO Caputo fractional derivative is not the convolution operator. Therefore, those efficient approximations for the CO Caputo fractional derivative cannot be directly applied to the VO Caputo fractional derivative.

In this paper, based on a linear combination of exponentials \cite{Beylkin-2017}, we develop a fast method for  approximating the VO Caputo fractional derivative (\ref{caputo-def}). Since the historical information of the integral requires much active memory and expensive computational operations, the main objective of the proposed fast method is to efficiently discretize the integral on the interval $[0, t-\Delta t]$ in a cheaper way, where $\Delta t$ is the time step. Unlike applying the sum-of-exponentials method \cite{Jiang-2017}, our strategy is to approximate the singular kernel in \eqref{caputo-def} directly \emph{without using the Laplace transform.} More precisely, the singular kernel in \eqref{caputo-def} is approached by the exponential-sum-approximation (ESA) method with the expected accuracy $\epsilon$, {\it where the weights of the sum are dependent on $t$, the quadrature exponents and the total number of the exponentials are independent of $t$}. Accordingly, a fast algorithm is devised to approach the VO Caputo fractional derivative and the computational complexity of the proposed algorithm is of $\mathcal{O}(n\log^2 n)$ with $\mO(\log^2n)$ active memory, which are much cheaper than $\mathcal{O}(n^2)$ cost and $\mathcal{O}(n)$ storage by the $L1$ approximation formula. Moreover, the local truncated error of the proposed approximation to the VO Caputo fractional derivative is estimated with $\mO(\epsilon + \Delta t^{2-\overline{\alpha}})$ accuracy. Based on this fast method, we construct a fast finite difference scheme (the fast ESA scheme) to solve the VO time-fractional diffusion equations. Meanwhile, the maximum principle is exploited to verify that the fast ESA scheme is unconditionally stable and convergent with the order of $\mathcal{O}(\epsilon+\Delta t^{2-\overline{\alpha}}+\Delta x^2)$, where $\Delta x$ is the spatial step. We remark that once the expected accuracy $\epsilon\leq \mO(\Delta t^{2-\overline{\alpha}})$ is reached, the proposed scheme can achieve the convergence order of $\mathcal{O}(\Delta t^{2-\overline{\alpha}}+\Delta x^2)$. Compared with the $L1$ scheme, the fast ESA scheme has the same accuracy and convergence order with less CPU time and memory.

The paper is organized as follows. In Section \ref{fast-approximation}, we propose a fast algorithm for approaching the VO Caputo fractional derivative utilizing an ESA technique and the local truncated error is studied. In Section \ref{finite-difference-scheme}, the fast method is applied to construct a fast ESA scheme to solve the VO time-fractional diffusion equations. The stability and convergence of the scheme are also investigated by the maximum principle. In Section \ref{numerical-results}, numerical results are reported to demonstrate the efficiency of the proposed method. Concluding remarks are given in Section \ref{concluding-remarks}.

\section{Fast approximation to VO Caputo fractional derivative}\label{fast-approximation}
In this section, we propose an ESA method to fast approximate the VO Caputo fractional derivative such that the memory and computational complexity are significantly reduced compared with the $L1$ approximation formula \cite{Shen-2012}.

Consider the VO Caputo fractional derivative defined by (\ref{caputo-def}) with $t\in(0,T]$ $(T\geq1)$. For a positive integer $n$, let $\Delta t=T/n$ and $t_k=k\Delta t$ for $k=0,1,\ldots,n$. At each time level $t_k$, we have
\begin{align}\label{caputo}
\prescript{C}{0}{\mathcal{D}}^{\alpha(t)}_tu(t)|_{t=t_k}&=\frac 1 {\Gamma(1-\alpha(t_k))}\int_{0}^{t_{k}}\frac{u'(\tau)}{(t_k-\tau)^{\alpha(t_k)}}d\tau.
\end{align}
For convenience, denote $\alpha_k=\alpha(t_k)$. To discretize the VO Caputo fractional derivative, we introduce the first-order approximation to $u'(\tau)$ over the interval $[t_{k-1},t_k]$ with $1\leq k\leq n$:
\begin{align*}
\Pi_1^ku'(\tau)=\frac{u(t_k)-u(t_{k-1})}{\Delta t},
\end{align*}
and the piecewise approximation function
\begin{align*}
\Pi_1 u'(\tau)=\left\{\Pi_1^ku'(\tau)|\tau\in[t_{k-1},t_k], k=1,2,\ldots,n\right\}.
\end{align*}
Based on the above interpolation polynomials, the $L1$ approximation formula to (\ref{caputo}) is obtained as\cite{Shen-2012}
\begin{align}\label{L1}
\prescript{L1}{0}{\mathcal{D}}^{\alpha_k}_tu(t_k)=&\frac 1 {\Gamma(1-\alpha_k)}\int_{0}^{t_k}\frac{\Pi_1u'(\tau)}{(t_k-\tau)^{\alpha_k}}d\tau\\
=&\frac {\Delta t^{-\alpha_k}}{\Gamma(2-\alpha_k)}\left[a_{k-1}^ku(t_k)-\sum\limits_{l=1}^{k-1}(a_l^k-a_{l-1}^k)u(t_l)-a_0^ku(t_0)\right],\nonumber
\end{align}
where $a_l^k=(k-l)^{1-\alpha_k}-(k-l-1)^{1-\alpha_k}$. The local truncated error of the $L1$ approximation formula for the CO Caputo fractional derivative with the order $0<\gamma<1$ is estimated by the following lemma.
\begin{lemma}\label{L-truncation-L1-CO}(see \cite{Gao-2011})
Suppose $u(t)\in C^2[0,t_k]$. Let the CO Caputo fractional derivative at $t_k$ be defined by $\prescript{C}{0}{\mathcal{D}}^{\gamma}_tu(t)|_{t=t_k}$ and the $L1$ approximation formula be defined by $\prescript{L1}{0}{\mathcal{D}}^{\gamma}_tu(t_k)$. Then, we have
\begin{align}\label{truncation-L1-CO}
\left|\prescript{C}{0}{\mathcal{D}}^{\gamma}_tu(t)|_{t=t_k}-\prescript{L1}{0}{\mathcal{D}}^{\gamma}_tu(t_k)\right|=\mathcal{O}(\Delta t^{2-\gamma}).
\end{align}
\end{lemma}

According to Lemma \ref{L-truncation-L1-CO}, by letting $\gamma=\alpha_k$, we have the following corollary to show the local truncated error of the $L1$ approximation formula for the VO Caputo fractional derivative.
\begin{corollary}\label{L-truncation-L1}
Suppose $u(t)\in C^2[0,t_k]$. Let the VO Caputo fractional derivative at $t_k$ be as in (\ref{caputo}) and the $L1$ approximation formula be as in (\ref{L1}). Then, we have
\begin{align}\label{truncation-L1}
\left|\prescript{C}{0}{\mathcal{D}}^{\alpha(t)} _tu(t)|_{t=t_k}-\prescript{L1}{0}{\mathcal{D}}^{\alpha_k}_tu(t_k)\right|=\mathcal{O}(\Delta t^{2-\alpha_k}).
\end{align}
\end{corollary}

Utilizing the $L1$ approximation formula to calculate the value at the current time level, it needs to compute the summation of a series including the values of all previous time levels. Therefore, the $L1$ approximation formula requires $\mathcal{O}(n)$ storage and $\mathcal{O}(n^2)$ computational complexity. It inspires us to construct a fast algorithm to approach the $L1$ approximation formula (\ref{L1}). Our idea is to approximate the singular kernel in \eqref{caputo-def} directly to obtain an ESA method. We introduce the following lemma, which is helpful for developing the fast algorithm.

\begin{lemma}\label{L-soe}(see \cite{Beylkin-2017})
For any constant $\gamma>0$, $0<\delta\leq t\leq1$, and $0<\epsilon\leq\ 1/e$, there exist a constant $h$, integers $\overline{N}$ and $\underline{N}$, which satisfy
\begin{align}\label{parameter}
h&\leq \frac {2\pi}{\log3+\gamma\log(\cos 1)^{-1}+\log\epsilon^{-1}},\nonumber\\
\underline{N}&\geq \frac 1{h}\left(\frac 1\gamma\log\epsilon+\frac 1\gamma\log\Gamma(1+\gamma)\right),\\
\overline{N}&\leq \frac 1{h}\left(\log\frac 1\delta+\log\log\frac 1\epsilon+\log\gamma+\frac 12\right),\nonumber
\end{align}
such that
\begin{eqnarray}\label{soe}
\left|t^{-\gamma}-\sum_{i=\underline{N}+1}^{\overline{N}}\theta_{\gamma,i} e^{-\lambda_it}\right|\leq t^{-\gamma}\epsilon,
\end{eqnarray}
where $\theta_{\gamma,i}=\frac{he^{\gamma ih}}{\Gamma(\gamma)}$ and $\lambda_i=e^{ih}$.
Furthermore, the total number of terms in the summation depends on the fixed power $\gamma$, the expected accuracy $\epsilon$, and the parameter $\delta$, which can be estimated as
\begin{align*}
N_\epsilon=\overline{N}-\underline{N}\leq \frac 1 {10}\left(2\log\frac 1 \epsilon+\log\gamma+2\right)\left(\log \frac 1 \delta+\frac 1 \gamma\log\frac 1 \epsilon+\log\log\frac 1 \epsilon+\frac 3 2\right).
\end{align*}
\end{lemma}

\subsection{Fast algorithm to approximate VO Caputo fractional derivative}\label{fast-algorithm}
To develop the fast algorithm, we first split the integral in (\ref{L1}) into two parts as follows
\begin{align}\label{HL-parts}
\prescript{L1}{0}{\mathcal{D}}^{\alpha_k}_tu(t_k)&=\frac 1 {\Gamma(1-\alpha_k)}\left[\int_{0}^{t_{k-1}}\frac{\Pi_1u'(\tau)}{(t_k-\tau)^{\alpha_k}}d\tau+\int_{t_{k-1}}^{t_k}\frac{\Pi_1u'(\tau)}{(t_k-\tau)^{\alpha_k}}d\tau\right]\nonumber\\
&\equiv \frac 1 {\Gamma(1-\alpha_k)}\left[\mathcal{I}_{\Delta t}^{0,k-1}(t_k)+\mathcal{I}_{\Delta t}^{k-1,k}(t_k)\right].
\end{align}
Since the second part $\mathcal{I}_{\Delta t}^{k-1,k}(t_k)$ contributes few memory and computational cost compared with the first part $\mathcal{I}_{\Delta t}^{0,k-1}(t_k)$, we calculate $\mathcal{I}_{\Delta t}^{k-1,k}(t_k)$ directly. 
Thus, the task becomes to approximate the integral on the interval $[0,t_{k-1}]$ in (\ref{HL-parts}) efficiently and accurately.
For $k=2,3,\ldots,n$, we have
\begin{align}\label{his}
\mathcal{I}_{\Delta t}^{0,k-1}(t_k)
=&T^{-\alpha_k}\int_{0}^{t_{k-1}}\Pi_1u'(\tau)\left(\frac{t_k-\tau}T\right)^{-\alpha_k}d\tau.
\end{align}
Note that $\alpha_k>0$ and $0<\frac{\Delta t}T\leq \frac{t_k-\tau}T\leq1$ for $\tau\in[0,t_{k-1}]$. With the help of Lemma \ref{L-soe}, the kernel $\left(\frac{t_k-\tau}T\right)^{-\alpha_k}$ in (\ref{his}) can be approached by an ESA technique; i.e.,
\begin{align}\label{his-app}
\mathcal{I}_{\Delta t}^{0,k-1}(t_k)\approx&T^{-\alpha_k}\sum_{i=\underline{N}+1}^{\overline{N}}\theta_{k,i}\int_0^{t_{k-1}}\Pi_1u'(\tau){e^{-\lambda_i(t_k-\tau)/T}}d\tau,
\end{align}
where the quadrature weights and exponents are defined as in (\ref{soe})
\begin{align*}
\theta_{k,i}=\frac{he^{\alpha_kih}}{\Gamma(\alpha_k)},\ \ \lambda_i=e^{ih},
\end{align*}
in which $h$, $\underline{N}$ and $\overline{N}$ are properly chosen based on (\ref{parameter}).

\begin{remark}
The VO Caputo fractional derivative implies different value $\alpha_k$ at different time level, hence using the sum-of-exponentials for the CO sense proposed in \cite{Jiang-2017} cannot be implemented since that requires different quadrature points and numbers of exponentials to save the history information at different time levels.
\end{remark}

To keep the quadrature exponents and the number of the exponentials in the summation unchanged at different time, there are some considerations for the choices of the expected accuracy $\epsilon$ and the parameters $h$, $\underline{N}$ and $\overline{N}$.

\begin{remark}\label{R-parameter-choose}
We suggest taking the expected accuracy $\epsilon\leq\mO(\Delta t^{2-\overline{\alpha}})$ (will be discussed in Section \ref{finite-difference-scheme}). Once the expected accuracy $\epsilon$ is given, according to Lemma \ref{L-soe}, the parameters $h$, $\underline{N}$ and $\overline{N}$ will be chosen to satisfy (\ref{parameter}) at each time level. In the numerical simulations performed in this work, for each time level, we always choose $\epsilon=\Delta t^2$ and
\begin{align}\label{parameter-choose}
h&=\frac {2\pi}{\log3+\overline{\alpha}\log(\cos 1)^{-1}+\log\epsilon^{-1}},\nonumber\\
\underline{N}&=\left\lceil\frac 1{h\underline{\alpha}}\left[\log\epsilon+\log\Gamma(1+\overline{\alpha})\right]\right\rceil,\\
\overline{N}&=\left\lfloor\frac 1{h}\left(\log \frac T{\Delta t}+\log\log\epsilon^{-1}+\log\underline{\alpha}+2^{-1}\right)\right\rfloor.\nonumber
\end{align}
\end{remark}

We now obtain the approximation for $\mathcal{I}_{\Delta t}^{0,k-1}(t_k)$. The related discretization formula will be obtained (see Lemma \ref{L-truncation-L1-fast} later) as
\begin{align}\label{his-L1-fast}
\mathcal{I}_{\Delta t}^{0,k-1}(t_k)=\mathcal{I}_{\Delta t,\epsilon}^{0,k-1}(t_k)+\mathcal{O}(\epsilon),
\end{align}
where
\begin{align}\label{his-fast}
\mathcal{I}_{\Delta t,\epsilon}^{0,k-1}(t_k)=T^{-\alpha_k}\sum_{i=\underline{N}+1}^{\overline{N}}\theta_{k,i}v_{k,i},
\end{align}
in which
\begin{align}\label{his-y}
v_{k,i}=\int_0^{t_{k-1}}\Pi_1u'(\tau)e^{-\lambda_i(t_k-\tau)/T}d\tau.
\end{align}

Next, we discuss how to implement the fast algorithm to approximate the VO Caputo fractional derivative. We note that $v_{1,i}=0$ and $v_{k,i}$ $(k=2,3,\ldots,n)$ can be exactly calculated by the following recursive formula:
\begin{align}\label{his-y-compute}
v_{k,i}=&e^{-\lambda_i\Delta t/T} v_{k-1,i}+\int_{t_{k-2}}^{t_{k-1}}\Pi_1^{k-1}u'(\tau)e^{-\lambda_i(t_k-\tau)/T}d\tau\nonumber\\
=&e^{-\lambda_i\Delta t/T} v_{k-1,i}+T\frac{e^{-\lambda_i\Delta t /T}-e^{-2\lambda_i\Delta t/T}}{\lambda_i\Delta t}\left[u(t_{k-1})-u(t_{k-2})\right].
\end{align}
Since $v_{k-1,i}$ is known at the current time level $t_k$, we compute $v_{k,i}$ directly.

Based on (\ref{his-L1-fast}), $\mathcal{I}_{\Delta t}^{0,k-1}(t_k)$ in (\ref{HL-parts}) can be replaced by (\ref{his-fast}) to approximate $\tensor*[^{L1}_0]{\mathcal{D}}{}^{\alpha_k}_tu(t_{k})$. Finally, we obtain the fast approximation formula for the VO Caputo fractional derivative as
\begin{align}\label{fast2}
\prescript{F}{0}{\mathcal{D}}^{\alpha_k}_tu(t_k)\equiv&\frac 1{\Gamma(1-\alpha_k)}\left[\mathcal{I}_{\Delta t,\epsilon}^{0,k-1}(t_k)+\mathcal{I}_{\Delta t}^{k-1,k}(t_k)\right]\nonumber\\
=&\frac {T^{-\alpha_k}}{\Gamma(1-\alpha_k)}\sum_{i=\underline{N}+1}^{\overline{N}}\theta_{k,i}v_{k,i}+\frac{u(t_{k})-u(t_{k-1})}{\Delta t^{\alpha_k}\Gamma(2-\alpha_k)},\ \ k=2,3,\ldots,n.
\end{align}
In addition, for $k=1$, the first part in (\ref{HL-parts}) satisfies $\mathcal{I}_{\Delta t}^{0,0}(t_1)=0$, thus
\begin{align}\label{fast1}
\prescript{F}{0}{\mathcal{D}}^{\alpha_1}_tu(t_1)&\equiv\frac 1 {\Gamma(1-\alpha_k)}\mathcal{I}_{\Delta t}^{0,1}(t_1)=\frac{u(t_1)-u(t_0)}{\Delta t^{\alpha_1}\Gamma(2-\alpha_1)}=\tensor*[^{L1}_0]{\mathcal{D}}{}^{\alpha_1}_tu(t_1).
\end{align}

Summarizing all this activity, we give the following algorithm to show the detailed instruction for the implementation of the fast algorithm for approximating the VO Caputo fractional derivative.

\begin{algorithm}[H]
\caption{The fast algorithm to approximate VO Caputo fractional derivative gradually}
\label{alg-fast}
\begin{algorithmic}[1]
\State Give the time step $\Delta t$, the expected accuracy $\epsilon$ and set $h,\underline{N},\overline{N}$ as in (\ref{parameter-choose})
\State Compute $\prescript{F}{0}{\mathcal{D}}_t^{\alpha_1}u(t_1)$ by (\ref{fast1})
\State Set $\left\{\lambda_i=e^{ih}\right\}_{i=\underline{N}+1}^{\overline{N}}$ and $\{v_{1,i}=0\}_{i=\underline{N}+1}^{\overline{N}}$
\For{$k=2,3,\ldots,n$}
\State Set $\left\{\theta_{k,i}=\frac{he^{\alpha_kih}}{\Gamma(\alpha_k)}\right\}_{i=\underline{N}+1}^{\overline{N}}$ and update $\{v_{k,i}\}_{i=\underline{N}+1}^{\overline{N}}$ by (\ref{his-y-compute}) 
\State Compute $\prescript{F}{0}{\mathcal{D}}_t^{\alpha_k}u(t_k)$ by (\ref{fast2}) using $\{\lambda_i,\theta_{k,i},v_{k,i}\}_{i=\underline{N}+1}^{\overline{N}}$
\EndFor
\end{algorithmic}
\end{algorithm}

\begin{remark}\label{R-memory}
The difficulty to discretize the VO Caputo fractional derivative is that the calculation of the value on the current time level needs to store the values of all previous time levels and compute the integrals on every sub-interval. As a result, the $L1$ approximation formula requires $\mathcal{O}(n)$ memory and $\mathcal{O}(n^2)$ computational cost, respectively. Significantly, the proposed fast algorithm can reduce the storage and computational complexity. Indeed, Lemma \ref{L-soe} shows that at each time level, the total number of exponentials in the ESA method is independent of time and can be bounded by
\begin{align}\label{memory}
N_\epsilon=\overline{N}-\underline{N}\leq \frac 1 {10}\left(2\log\frac 1 \epsilon+\log\overline{\alpha}+2\right)\left(\log \frac T {\Delta t}+\frac 1 {\underline{\alpha}} \log\frac 1 \epsilon+\log\log\frac 1 \epsilon+\frac 3 2\right),
\end{align}
which is of $\mathcal{O}(\log^2n)$ with the expected accuracy $\epsilon\leq\mO(\Delta t^{2-\overline{\alpha}})$, when $n$ is sufficiently large. Meanwhile, at each time level, it only needs $\mathcal{O}(1)$ computational cost to compute $v_{k,i}$ since $v_{k-1,i}$ is known at that point. In total the fast algorithm requires only $\mathcal{O}(\log^2n)$ memory and $\mathcal{O}(n\log^2n)$ computational complexity when numerically discretize the VO Caputo fractional derivative. The proposed fast method provides an efficient tool to approximate the VO Caputo fractional derivative, which can be applied to simulate the VO time-fractional diffusion equations (\ref{VOTFDE1})--(\ref{VOTFDE3}).
\end{remark}

\subsection{Local truncated error of the fast approximation}\label{truncation}
To investigate the local truncated error of the fast approximation formula (\ref{fast2})--(\ref{fast1}) to the VO Caputo fractional derivative $\prescript{C}{0}{\mathcal{D}}^{\alpha(t_k)}_tu(t)|_{t=t_k}$ , we give the following lemma to state the error bound of the fast approximation formula $\prescript{F}{0}{\mathcal{D}}^{\alpha_k}_tu(t_k)$ to the $L1$ approximation formula $\prescript{L1}{0}{\mathcal{D}}^{\alpha_k}_tu(t_k)$.
\begin{lemma}\label{L-truncation-L1-fast}
Suppose $u(t)\in C^2[0,t_k]$. Let the $L1$ approximation formula be as in (\ref{L1}), the fast approximation formula be defined by (\ref{fast2})--(\ref{fast1}) and $\epsilon$ be the expected accuracy. Then, we have
\begin{align}\label{truncation-L1-fast}
\left|\prescript{L1}{0}{\mathcal{D}}^{\alpha_k}_tu(t_k)-\prescript{F}{0}{\mathcal{D}}^{\alpha_k}_tu(t_k)\right|=\mathcal{O}(\epsilon).
\end{align}
\end{lemma}
\begin{proof}
Clearly, (\ref{fast1}) implies the lemma is valid for $k=1$. For $k=2,3,\ldots,n$, obviously, the only difference between $\prescript{F}{0}{\mathcal{D}}^{\alpha_k}_tu(t_k)$  and $\prescript{L1}{0}{\mathcal{D}}^{\alpha_k}_tu(t_k)$ is the calculation of the first part $\mathcal{I}_{\Delta t}^{0,k-1}(t_k)$ in (\ref{HL-parts}); i.e.,
\begin{align}\label{truncation-L1-fast-proof}
\left|\prescript{L1}{0}{\mathcal{D}}^{\alpha_k}_tu(t_k)-\prescript{F}{0}{\mathcal{D}}^{\alpha_k}_tu(t_k)\right|=\frac 1 {\Gamma(1-\alpha_k)}\left|\mathcal{I}_{\Delta t}^{0,k-1}(t_k)-\mathcal{I}_{\Delta t,\epsilon}^{0,k-1}(t_k)\right|.
\end{align}
According to (\ref{his})--(\ref{his-app}), the error between $\mathcal{I}_{\Delta t}^{0,k-1}(t_k)$ and $\mathcal{I}_{\Delta t,\epsilon}^{0,k-1}(t_k)$ is obtained as
\begin{align*}
\left|\mathcal{I}_{\Delta t}^{0,k-1}(t_k)-\mathcal{I}_{\Delta t,\epsilon}^{0,k-1}(t_k)\right|
=&T^{-\alpha_k}\left|\int_{0}^{t_{k-1}}\Pi_1u'(\tau)\left[\left(\frac{t_k-\tau}T\right)^{-\alpha_k}-\sum\limits_{i=\underline{N}+1}^{\overline{N}}\theta_{k,i}e^{-\lambda_i(t_k-\tau)/T}\right]d\tau\right|.
\end{align*}
By Lemma \ref{L-soe}, it follows that
\begin{align*}
(1-\epsilon)\left(\frac{t_k-\tau}T\right)^{-\alpha_k}\leq \sum_{i=\underline{N}+1}^{\overline{N}}\theta_{k,i}e^{-\lambda_i(t_k-\tau)/T} \leq(1+\epsilon)\left(\frac{t_k-\tau}T\right)^{-\alpha_k}.
\end{align*}
Therefore,
\begin{align*}
\left|\mathcal{I}_{\Delta t}^{0,k-1}(t_k)-\mathcal{I}_{\Delta t,\epsilon}^{0,k-1}(t_k)\right|\leq&\epsilon\left|\int_{0}^{t_{k-1}}\Pi_1u'(\tau)(t_k-\tau)^{-\alpha_k}d\tau\right|\\
\leq&\epsilon\max_{0\leq t\leq t_{k-1}}\left|u'(t)\right|\int_{0}^{t_{k-1}}(t_k-\tau)^{-\alpha_k}d\tau\\
\leq&\frac {\epsilon}{1-\alpha_k}\max_{0\leq t\leq t_{k-1}}\left|u'(t)\right|t_k^{1-\alpha_k}.
\end{align*}
Substituting it into (\ref{truncation-L1-fast-proof}) implies (\ref{truncation-L1-fast}). The proof is completed.
\end{proof}

Finally, we immediately obtain the following theorem to estimate the local truncated error of the fast approximation formula $\prescript{F}{0}{\mathcal{D}}^{\alpha_k}_tu(t_k)$ to the VO Caputo farctional derivative $\prescript{C}{0}{\mathcal{D}}^{\alpha(t)}_tu(t)|_{t=t_k}$.
\begin{theorem}\label{T-truncation-fast}
Suppose $u(t)\in C^2[0,t_k]$. Let the VO Caputo fractional derivative at $t_k$ be as in (\ref{caputo}), its fast approximation formula be defined by (\ref{fast2})--(\ref{fast1}) and $\epsilon$ be the expected accuracy. Then, we have
\begin{align*}
\left|\prescript{C}{0}{\mathcal{D}}^{\alpha(t)}_tu(t)|_{t=t_k}-\prescript{F}{0}{\mathcal{D}}^{\alpha_k}_tu(t_k)\right|=\mathcal{O}\left(\Delta t^{2-\alpha_k}+\epsilon\right).
\end{align*}
\end{theorem}
\begin{proof}
The triangle inequality leads to
\begin{align*}
\left|\prescript{C}{0}{\mathcal{D}}^{\alpha(t)}_tu(t)|_{t=t_k}-\prescript{F}{0}{\mathcal{D}}^{\alpha_k}_tu(t_k)\right|\leq &\left|\prescript{C}{0}{\mathcal{D}}^{\alpha(t)}_tu(t)|_{t=t_k}-\prescript{L1}{0}{\mathcal{D}}^{\alpha_k}_tu(t_k)\right|+\left|\prescript{L1}{0}{\mathcal{D}}^{\alpha_k}_tu(t_k)-\prescript{F}{0}{\mathcal{D}}^{\alpha_k}_tu(t_k)\right|.
\end{align*}
The desired result now follows on recalling (\ref{truncation-L1}) and (\ref{truncation-L1-fast}).
\end{proof}

\section{Fast ESA scheme for VO time-fractional diffusion equations}\label{finite-difference-scheme}
In this section, we apply the fast method proposed in Section \ref{fast-approximation} to construct a fast finite difference scheme (the fast ESA scheme) for solving the VO time-fractional diffusion equations (\ref{VOTFDE1})--(\ref{VOTFDE3}).

Let $m$ be a positive integer, $\Delta x=x_R/m$ be the size of spatial grid, and define a spatial partition $x_j=j\Delta x$ for $j=0,1,\ldots,m$. We introduce the notation
\begin{align*}
\delta_x^2 u(x_j,t_k)=\frac{u(x_{j-1},t_k)-2u(x_j,t_k)+u(x_{j+1},t_k)}{\Delta x^2}.
\end{align*}
Consequently,
\begin{align}\label{spatial-derivative}
\frac{\partial^2u(x_j,t_k)}{\partial x^2}=\delta_x^2 u(x_j,t_k)+\mathcal{O}(\Delta x^2).
\end{align}
Denote $u_j^k$ be the approximate solution to $u(x_j,t_k)$, $f_j^k=f(x_j,t_k)$ and $\varphi_j=\varphi(x_j)$.
The $L1$ scheme is obtained as (see \cite{Shen-2012})
\begin{align}
&\prescript{L1}{0}{\mathcal{D}}^{\alpha_k}_tu_j^k=\delta_x^2u_j^k+f_j^k,\ \ j=1,2,\ldots,m-1,\ \ k=1,2,\ldots,n,\label{L1-scheme1}\\
&u_j^0=\varphi_j,\ \ j=1,2,\ldots,m-1,\label{L1-scheme2}\\
&u_0^k=u_m^k=0,\ \ k=1,2,\ldots,n,\label{L1-scheme3}
\end{align}
where $\tensor*[^{L1}_0]{\mathcal{D}}{}^{\alpha_k}_tu_j^k$ is defined by (\ref{L1}).

\subsection{Preliminary work}
Using (\ref{fast2}) and the $L1$ scheme (\ref{L1-scheme1}), we propose the fast ESA scheme for (\ref{VOTFDE1})--(\ref{VOTFDE3}) as
\begin{align}
&\prescript{F}{0}{\mathcal{D}}^{\alpha_k}_tu_j^k=\delta_x^2u_j^k+f_j^k,\ \ j=1,2,\ldots,m-1,\ \ k=1,2,\ldots,n,\label{fast-scheme1}\\
&u_j^0=\varphi_j,\ \ j=1,2,\ldots,m-1,\label{fast-scheme2}\\
&u_0^k=u_m^k=0,\ \ k=1,2,\ldots,n,\label{fast-scheme3}
\end{align}
where $\prescript{F}{0}{\mathcal{D}}^{\alpha_k}_tu_j^k$ is rewritten as
\begin{align}\label{fast}
\prescript{F}{0}{\mathcal{D}}^{\alpha_k}_tu_j^k=&\frac {\Delta t^{-\alpha_k}}{\Gamma(2-\alpha_k)}u_j^k-\frac {\Delta t^{-\alpha_k}}{\Gamma(2-\alpha_k)}\sum_{l=1}^{k-1}b_l^ku_j^l-\frac {\Delta t^{-\alpha_k}}{\Gamma(2-\alpha_k)}b_0^ku_j^0,
\end{align}
where
\begin{align*}
&b_0^k=T^{-\alpha_k}\Delta t^{\alpha_k-1}(1-\alpha_k)\sum_{i=\underline{N}+1}^{\overline{N}}\theta_{k,i}\int_{t_0}^{t_1}e^{-\lambda_i(t_k-\tau)/T} d\tau,\\
&b_{k-1}^k=1-T^{-\alpha_k}\Delta t^{\alpha_k-1}(1-\alpha_k)\sum_{i=\underline{N}+1}^{\overline{N}}\theta_{k,i}\int_{t_{k-2}}^{t_{k-1}}e^{-\lambda_i(t_k-\tau)/T} d\tau,
\end{align*}
and
\begin{align*}
b_l^k=&T^{-\alpha_k}\Delta t^{\alpha_k-1}(1-\alpha_k)\sum_{i=\underline{N}+1}^{\overline{N}}\theta_{k,i} \left[\int_{t_l}^{t_{l+1}}e^{-\lambda_i(t_k-\tau)/T} d\tau-\int_{t_{l-1}}^{t_l}e^{-\lambda_i(t_k-\tau)/T} d\tau\right]
\end{align*}
for $l=1,2,\ldots,k-2$. In particular $b_0^1=1$.

For convenience, we now denote \begin{align*}
s_k=\frac {\Delta x^2\Delta t^{-\alpha_k}}{\Gamma(2-\alpha_k)}.
\end{align*}
Then, (\ref{fast-scheme1}) can be written as
\begin{align}\label{re-scheme1}
-u_{j+1}^k+(2+s_k)u_j^k-u_{j-1}^k=s_k\sum_{l=1}^{k-1}b_l^ku_j^l+s_kb_0^ku_j^0+\Delta x^ 2f_j^k.
\end{align}

Firstly, we give the following lemma for those coefficients in (\ref{fast}), which plays a vital role in the investigation of the stability and convergence of the fast ESA scheme (\ref{fast-scheme1})--(\ref{fast-scheme3}).
\begin{lemma}\label{L-coe}
Let $\{b_l^k\}_{l=0}^{k-1}$ $(k=1,2,\ldots,n)$ be defined by (\ref{fast}) and $\epsilon$ be the expected accuracy. Denote
\begin{align*}
\epsilon_k=
\left\{
  \begin{array}{ll}
    0, & {k=1,} \vspace{1mm}\\
    \frac{\epsilon}{1+\epsilon}T^{-\alpha_k}\Delta t^{\alpha_k-1}(1-\alpha_k)\sum\limits_{i=\underline{N}+1}^{\overline{N}}\theta_{k,i}\int_{t_{k-2}}^{t_{k-1}}e^{-\lambda_i(t_k-\tau)/T} d\tau, & {k=2,3,\ldots,n.}
  \end{array}
\right.
\end{align*}
Then, we have
 \begin{description}
   \item[i).] $\{b_l^k\}_{l=0}^{k-2}>0,\ \ k=2,3,\ldots,n$;
   \item[ii).] $b_{k-1}^k+\epsilon_k>0,\ \ k=1,2,\ldots,n$;
   \item[iii).] $|b_{k-1}^k|\leq b_{k-1}^k+2\epsilon_k,\ \ k=1,2,\ldots,n$;
   \item[iv).] $\sum\limits_{l=0}^{k-1}b_l^k+2\epsilon_k\leq1+2\epsilon,\ \ k=1,2,\ldots,n$.
 \end{description}
\end{lemma}
\begin{proof}
Obviously, the lemma is valid for $k=1$. It remains to prove the lemma for $k\geq2$.

\textbf{i).} The result is obtained by a straight forward calculation.

\textbf{ii).} By the definition of $b_{k-1}^k$ and $\epsilon_k$, we have
\begin{align*}
b_{k-1}^k+\epsilon_k=&1-\frac 1{1+\epsilon}T^{-\alpha_k}\Delta t^{\alpha_k-1}(1-\alpha_k)\sum_{i=\underline{N}+1}^{\overline{N}}\theta_{k,i}\int_{t_{k-2}}^{t_{k-1}}e^{-\lambda_i(t_k-\tau)/T} d\tau\\
\geq&1-\Delta t^{\alpha_k-1}(1-\alpha_k)\int_{t_{k-2}}^{t_{k-1}}(t_k-\tau)^{-\alpha_k} d\tau\\
=&2-2^{1-\alpha_k}>0.
\end{align*}

\textbf{iii).} The triangle inequality and \textbf{ii).} imply that
\begin{align*}
\left|b_{k-1}^k\right|=\left|b_{k-1}^k+\epsilon_k-\epsilon_k\right|\leq\left|b_{k-1}^k+\epsilon_k\right|+\left|\epsilon_k\right|=&b_{k-1}^k+2\epsilon_k.
\end{align*}

\textbf{iv).} Summing up $b_l^k$ for $l$ from $0$ to $k-1$, and rearranging the integral terms, we have
\begin{align*}
\sum_{l=0}^{k-1}b_l^k=&T^{-\alpha_k}\Delta t^{\alpha_k-1}(1-\alpha_k)\sum_{i=\underline{N}+1}^{\overline{N}}\theta_{k,i}\int_{t_0}^{t_1}e^{-\lambda_i(t_k-\tau)/T} d\tau\\
&+T^{-\alpha_k}\Delta t^{\alpha_k-1}(1-\alpha_k) \sum_{l=1}^{k-2}\left\{\sum_{i=\underline{N}+1}^{\overline{N}}\theta_{k,i}\left[\int_{t_l}^{t_{l+1}}e^{-\lambda_i(t_k-\tau)/T} d\tau-\int_{t_{l-1}}^{t_l}e^{-\lambda_i(t_k-\tau)/T} d\tau\right]\right\}\\
&+1-T^{-\alpha_k}\Delta t^{\alpha_k-1}(1-\alpha_k)\sum_{i=\underline{N}+1}^{\overline{N}}\theta_{k,i} \int_{t_{k-2}}^{t_{k-1}}e^{-\lambda_i(t_k-\tau)/T} d\tau\\
=&1.
\end{align*}
Therefore, we obtain
\begin{align*}
\sum_{l=0}^{k-1}b_l^k+2\epsilon_k=&1+2\frac{\epsilon}{1+\epsilon}T^{-\alpha_k}\Delta t^{\alpha_k-1}(1-\alpha_k)\sum_{i=\underline{N}+1}^{\overline{N}}\theta_{k,i}\int_{t_{k-2}}^{t_{k-1}}e^{-\lambda_i(t_k-\tau)/T} d\tau\\
\leq&1+2\epsilon\Delta t^{\alpha_k-1}(1-\alpha_k)\int_{t_{k-2}}^{t_{k-1}}(t_k-\tau)^{-\alpha_k} d\tau\\
=&1+2\epsilon(2^{1-\alpha_k}-1)\\
\leq&1+2\epsilon.
\end{align*}
The proof of the lemma is completed.
\end{proof}

\subsection{Stability}\label{stability}
Now, we discuss the stability of the fast ESA scheme (\ref{fast-scheme1})--(\ref{fast-scheme3}). To investigate the stability, we denote
\begin{align*}
u^k=\left[u_1^k,u_2^k,\ldots,u_{m-1}^k\right]^\tT,\ \ f^k=\left[f_1^k,f_2^k,\ldots,f_{m-1}^k\right]^\tT,\ \ k=0,1,\ldots,n,
\end{align*}
and
\begin{align*}
c_f=\max\limits_{1\leq k\leq n}\|f^k\|_\infty,\ \ c_\gamma=\max\limits_{1\leq k\leq n}\Gamma(1-\alpha_k).
\end{align*}
According to (\ref{re-scheme1}) and Lemma \ref{L-coe}, we have
\begin{align*}
\left|(2+s_k)u_j^k\right|\leq&\left|u_{j+1}^k\right|+\left|u_{j-1}^k\right|+s_k(b_{k-1}^k+2\epsilon_k)\left|u_j^{k-1}\right|+s_k\sum_{l=1}^{k-2}b_l^k\left|u_j^l\right|\\
&+s_kb_0^k\left|u_j^0\right|+\Delta x^ 2\left|f_j^k\right|,\ \ k=2,3,\ldots,n.
\end{align*}
Therefore,
\begin{align}\label{norm-stability1}
\|u^k\|_\infty\leq&(b_{k-1}^k+2\epsilon_k)\|u^{k-1}\|_\infty+\sum_{l=1}^{k-2}b_l^k\|u^l\|_\infty\nonumber\\
&+b_0^k\left(\|u^0\|_\infty+\frac{\Delta x^2}{b_0^ks_k}\|f^k\|_\infty\right),\ \ k=2,3,\ldots,n.
\end{align}
Similarly, when $k=1$, we have
\begin{align}\label{norm-stability2}
\|u^1\|_\infty\leq b_0^1\left(\|u^0\|_\infty+\frac{\Delta x^2}{b_0^1s_1}\|f^1\|_\infty\right).
\end{align}

\begin{theorem}\label{T-stability}
Suppose the expected accuracy $\epsilon\leq\mO(\Delta t^{2-\overline{\alpha}})$ and $\{u_j^k|0\leq j\leq m, 0\leq k\leq n\}$ is the solution of (\ref{fast-scheme1})--(\ref{fast-scheme3}). Then, we have
\begin{align*}
\|u^k\|_\infty\leq e^T\|u^0\|_\infty+c_fc_\gamma\frac {e^T}{1-\epsilon}T^{\overline{\alpha}},\ \ k=1,2,\ldots,n.
\end{align*}
\end{theorem}
\begin{proof}
Note (\ref{norm-stability1})--(\ref{norm-stability2}), by a straight forward calculation, we have
\begin{align*}
b_0^k\geq(1-\epsilon)\left[k^{1-\alpha_k}-(k-1)^{1-\alpha_k}\right]\geq(1-\epsilon)(1-\alpha_k)k^{-\alpha_k},\ \ k=2,3,\ldots,n.
\end{align*}
Note that $b_0^1=1\geq(1-\epsilon)(1-\alpha_1)$. Thus
\begin{align*}
\frac{\Delta x^2}{b_0^ks_k}\leq t_k^{\alpha_k}\frac{\Gamma(1-\alpha_k)}{1-\epsilon},\ \ k=1,2,\ldots,n.
\end{align*}
Substituting it into (\ref{norm-stability1})--(\ref{norm-stability2}), we have
\begin{align}\label{re-norm-stability1}
\|u^k\|_\infty\leq&(b_{k-1}^k+2\epsilon_k)\|u^{k-1}\|_\infty+\sum_{l=1}^{k-2}b_l^k\|u^l\|_\infty\nonumber\\
&+b_0^k\left[\|u^0\|_\infty+t_k^{\alpha_k}\frac{\Gamma(1-\alpha_k)}{1-\epsilon}\|f^k\|_\infty\right],\ \ k=2,3,\ldots,n,
\end{align}
and
\begin{align}\label{re-norm-stability2}
\|u^1\|_\infty\leq b_0^1\left[\|u^0\|_\infty+t_1^{\alpha_1}\frac{\Gamma(1-\alpha_1)}{1-\epsilon}\|f^1\|_\infty\right].
\end{align}
Next the mathematical induction will be used on $k$ to prove that
\begin{align}\label{stability-induction}
\|u^k\|_\infty\leq&(1+2\epsilon)^k\|u^0\|_\infty+c_fc_\gamma\frac{(1+2\epsilon)^k}{1-\epsilon}T^{\overline{\alpha}}.
\end{align}
For the case $k=1$, in view of (\ref{re-norm-stability2}) and $b_0^1=1$, we have
\begin{align*}
\|u^1\|_\infty\leq&b_0^1\left[\|u^0\|_\infty+t_1^{\alpha_1}\frac{\Gamma(1-\alpha_1)}{1-\epsilon}\|f^1\|_\infty\right]\\
\leq&(1+2\epsilon)\|u^0\|_\infty+c_fc_\gamma\frac{1+2\epsilon}{1-\epsilon}T^{\overline{\alpha}},
\end{align*}
which implies (\ref{stability-induction}) is valid for $k=1$. Suppose that
\begin{align}\label{induction}
\|u^{k_0}\|_\infty\leq (1+2\epsilon)^{k_0}\|u^0\|_\infty+c_fc_\gamma\frac{(1+2\epsilon)^{k_0}}{1-\epsilon}T^{\overline{\alpha}}
,\ \ k_0=1,2,\ldots,k-1.
\end{align}
Combing with (\ref{re-norm-stability1}), (\ref{induction}), and Lemma \ref{L-coe}, we obtain
\begin{align*}
\|u^k\|_\infty
\leq&(b_{k-1}^k+2\epsilon_k)\left[(1+2\epsilon)^{k-1}\|u^0\|_\infty+c_fc_\gamma\frac{(1+2\epsilon)^{k-1}}{1-\epsilon}T^{\overline{\alpha}}\right]\\
&+\sum_{l=1}^{k-2}b_l^k\left[(1+2\epsilon)^l\|u^0\|_\infty+c_fc_\gamma\frac{(1+2\epsilon)^l}{1-\epsilon}T^{\overline{\alpha}}\right]\\
&+b_0^k\left[\|u^0\|_\infty+t_k^{\alpha_k}\frac{\Gamma(1-\alpha_k)}{1-\epsilon}\|f^k\|_\infty\right]\\
\leq&(1+2\epsilon)^{k-1}\left(\sum_{l=0}^{k-1}b_l^k+2\epsilon_k\right)\|u^0\|_\infty+c_fc_\gamma\frac{(1+2\epsilon)^{k-1}}{1-\epsilon}\left(\sum_{l=0}^{k-1}b_l^k+2\epsilon_k\right)T^{\overline{\alpha}}\\
\leq&(1+2\epsilon)^k\|u^0\|_\infty+c_fc_\gamma\frac{(1+2\epsilon)^k}{1-\epsilon}T^{\overline{\alpha}}.
\end{align*}
By the principle of induction, (\ref{stability-induction}) is valid for $k=1,2,\ldots,n$. Noticing the fact that $\epsilon\leq\mO(\Delta t^{2-\overline{\alpha}})$, we obtain
\begin{align*}
\|u^k\|_\infty\leq&(1+2\epsilon)^n\|u^0\|_\infty+c_fc_\gamma\frac{(1+2\epsilon)^n}{1-\epsilon}T^{\overline{\alpha}}\\
\leq&e^T\|u^0\|_\infty+c_fc_\gamma\frac{e^T}{1-\epsilon}T^{\overline{\alpha}}.
\end{align*}
The proof is completed.
\end{proof}

Theorem \ref{T-stability} reveals the stability of the fast ESA scheme (\ref{fast-scheme1})--(\ref{fast-scheme3}) with respect to the initial value and the source term. The theorem is also used to study the convergence of this scheme.
\subsection{Convergence}\label{Convergence}
The task is to investigate the convergence of the fast ESA scheme (\ref{fast-scheme1})--(\ref{fast-scheme3}). According to Theorem \ref{T-truncation-fast} and (\ref{spatial-derivative}), we have
\begin{align}\label{fast-scheme1-exact}
-u(x_{j+1},t_k)+\left(2+s_k\right)u(x_j,t_k)-u(x_{j-1},t_k)=&s_k\sum_{l=1}^{k-1}b_l^ku(x_j,t_l)+s_kb_0^ku(x_j,t_0)\nonumber\\
&+\Delta x^2f(x_j,t_k)+\Delta x^2r_j^k,
\end{align}
where
\begin{align}\label{L-r}
r_j^k=\mathcal{O}\left(\epsilon+\Delta t^{2-\alpha_k}+\Delta x^2\right).
\end{align}

Denote $E_{j}^k=u(x_j,t_k)-u_j^k$ for $j=1,2,\ldots,m-1$, $k=0,1,\ldots,n$ and
\begin{align*}
E^k=\left[E_1^k,E_2^k,\ldots,E_{m-1}^k\right]^\tT,\ \ r^k=\left[r_1^k,r_2^k,\ldots,r_{m-1}^k\right]^\tT,\ \ c_r=\max\limits_{1\leq k\leq n}\|r^k\|_\infty.
\end{align*}

Subtracting (\ref{re-scheme1}) from (\ref{fast-scheme1-exact}), we obtain the error equation for $j=1,2,\ldots,m-1$, $k=1,2,\ldots,n$
\begin{align}
&-E_{j+1}^k+\left(2+s_k\right)E_j^k-E_{j-1}^k=s_k\sum_{l=1}^{k-1}b_l^kE_j^l+s_kb_0^kE_j^0+\Delta x^2r_j^k,\label{error1}
\end{align}
with
\begin{align}
&E_j^0=0,\ \ j=1,2,\ldots,m-1,\label{error2}\\
&E_0^k=E_m^k=0,\ \ k=1,2,\ldots,n.\label{error3}
\end{align}
The following theorem states the convergence of the fast ESA scheme.
\begin{theorem}\label{T-convergence}
Suppose the expected accuracy $\epsilon\leq\mathcal{O}(\Delta t^{2-\overline{\alpha}})$, $u(x,t)\in C_{x,t}^{4,2}([0,x_R]\times[0,T])$ and $\{u_j^k|0\leq j\leq m, 0\leq k\leq n\}$ are solutions of the problem (\ref{VOTFDE1})--(\ref{VOTFDE3}) and the fast ESA scheme (\ref{fast-scheme1})--(\ref{fast-scheme3}), respectively. Let $E_j^k=u(x_j,t_k)-u_j^k$. Then, we have
\begin{align*}
\|E^k\|_\infty=\mathcal{O}\left(\epsilon+\Delta t^{2-\overline{\alpha}}+\Delta x^2\right),\ \ k=1,2,\ldots,n.
\end{align*}
\end{theorem}
\begin{proof}
Applying Theorem \ref{T-stability} for the error equation (\ref{error1})--(\ref{error3}), and noting $\|E^0\|_\infty=0$, we have
\begin{align*}
\|E^k\|_\infty\leq c_rc_\gamma\frac{e^T}{1-\epsilon}T^{\overline{\alpha}}.
\end{align*}
By (\ref{L-r}), there exists a positive constant $c$ such that
\begin{align*}
c_r&\leq c(\epsilon+\Delta t^{2-\overline{\alpha}}+\Delta x^2).
\end{align*}
Consequently, we obtain
\begin{align}\label{order}
\|E^k\|_\infty\leq cc_\gamma\frac{e^T}{1-\epsilon}T^{\overline{\alpha}}(\epsilon+\Delta t^{2-\overline{\alpha}}+\Delta x^2).
\end{align}
The proof is completed.
\end{proof}
\begin{remark}\label{R-epsilon}
From Theorem \ref{T-stability} and Theorem \ref{T-convergence}, the fast ESA scheme (\ref{fast-scheme1})--(\ref{fast-scheme3}) is stable and convergent with the order of $\mathcal{O}\left(\Delta t^{2-\overline{\alpha}}+\Delta x^2\right)$ with $\epsilon\leq\mO(\Delta t ^{2-\overline{\alpha}})$. In order to verify the correctness of the theoretical analysis, we choose $\epsilon=\Delta t^2$ in the actually computing.
\end{remark}

\section{Numerical results} \label{numerical-results}
In this section, two numerical examples are presented to verify the effectiveness of the fast ESA scheme (\ref{fast-scheme1})--(\ref{fast-scheme3}) compared with the $L1$ scheme (\ref{L1-scheme1})--(\ref{L1-scheme3}). All experiments are performed based on Matlab 2016b on a laptop with the configuration: Intel(R) Core(TM) i7-7500U CPU 2.70GHz and 8.00 GB RAM.

Denote
\begin{eqnarray*}
Err(\Delta x,\Delta t)=\|E^n\|_\infty,\ \ Order_t=\log_2\frac{Err(\Delta x,\Delta t)}{Err(\Delta x,\Delta t/2)}.
\end{eqnarray*}

\begin{example}\label{example1}
To verify the efficiency of our fast algorithm for the VO Caputo fractional derivative, we first solve an ordinary differential equation
\begin{align*}
\prescript{C}{0}{\mathcal{D}}^{\alpha(t)}_tu(x,t)=\frac{2}{\Gamma(3-\alpha(t))}t^{2-\alpha(t)}
\end{align*}
with the exact solution $u(t)=t^2$. The time interval is $[0,T]=[0,1]$.

Here two different types of VO functions $\alpha(t)=\frac{2+\sin(5t)}4$ and $\alpha(t)=1-0.8t$ are chosen. In order to match the accuracy as the $L1$ scheme, we set the expected accuracy $\epsilon=\Delta t^2$. Take the verify temporal step size $\Delta t$ from $1/10 000$ to $1/160 000$ and `$N_\epsilon$' is the total number of the exponentials in the ESA technique.
\end{example}


\begin{table}[t]
\setlength{\abovecaptionskip}{10pt}
 \setlength{\belowcaptionskip}{8pt}
 \renewcommand{\arraystretch}{1.25}
\begin{center}
\caption{Convergence rates and the CPU time, memory of the $L1$ scheme and the fast ESA scheme for Example \ref{example1} with $\epsilon=\Delta t^2$.}
\label{T1}
\def\temptablewidth{1.0\textwidth}
{\rule{\temptablewidth}{0.6pt}}
{\footnotesize
\begin{tabular*}{\temptablewidth}{@{\extracolsep{\fill}}cccccccccccccc}
 & &\multicolumn{4}{c}{{$L1$\ \ scheme}} & \multicolumn{4}{c}{{Fast ESA scheme}} \\
 \cline{2-5}\cline{6-11}
$\alpha(t)$ &$n$  &$Err(\Delta x,\Delta t)$      &$Order_t$  &CPU(s)     &Memory   &$Err(\Delta x,\Delta t)$        &$Order_t$ &CPU(s)   &Memory  &$N_{\epsilon}$
  \\
\hline
$\frac{2+\sin(5t)}4$  &10000  &7.4332e-7    & -       &4.45     &2.40e+5   &7.5299e-7   &-     &0.62  &1.28e+4 &264\\
                      &20000  &2.9880e-7    &1.31     &16.71    &4.80e+5   &3.0236e-7   &1.32  &1.02  &1.46e+4 &302\\
                      &40000  &1.2059e-7    &1.31     &64.54    &9.60e+5   &1.2084e-7   &1.32  &2.33  &1.67e+4 &345\\
                      &80000  &4.8831e-8    &1.30     &261.88   &1.92e+6   &4.8586e-8   &1.31  &4.56  &1.88e+4 &389\\
                      &160000 &1.9830e-8    &1.30     &1172.84  &3.84e+6   &2.0591e-8   &1.24  &9.02  &2.10e+4 &436\\
\hline
$1-0.8t$              &10000  &2.5409e-6    & -       &4.19     &2.40e+5   &2.5500e-6   &-     &0.60  &1.61e+4 &333\\
                      &20000  &1.1831e-6    &1.10     &16.04    &4.80e+5   &1.1893e-6   &1.10  &1.16  &1.84e+4 &382\\
                      &40000  &5.5390e-7    &1.09     &62.93    &9.60e+5   &5.6279e-7   &1.08  &2.39  &2.09e+4 &434\\
                      &80000  &2.6050e-7    &1.09     &255.47   &1.92e+6   &2.6181e-7   &1.10  &4.82  &2.36e+4 &490\\
                      &160000 &1.2299e-7    &1.08     &1143.08  &3.84e+6   &1.2434e-7   &1.07  &10.43 &2.65e+4 &549\\
\end{tabular*} {\rule{\temptablewidth}{1pt}}
}
\end{center}
\end{table}
\begin{figure}
\centering
 {\includegraphics[width=0.45\textwidth]{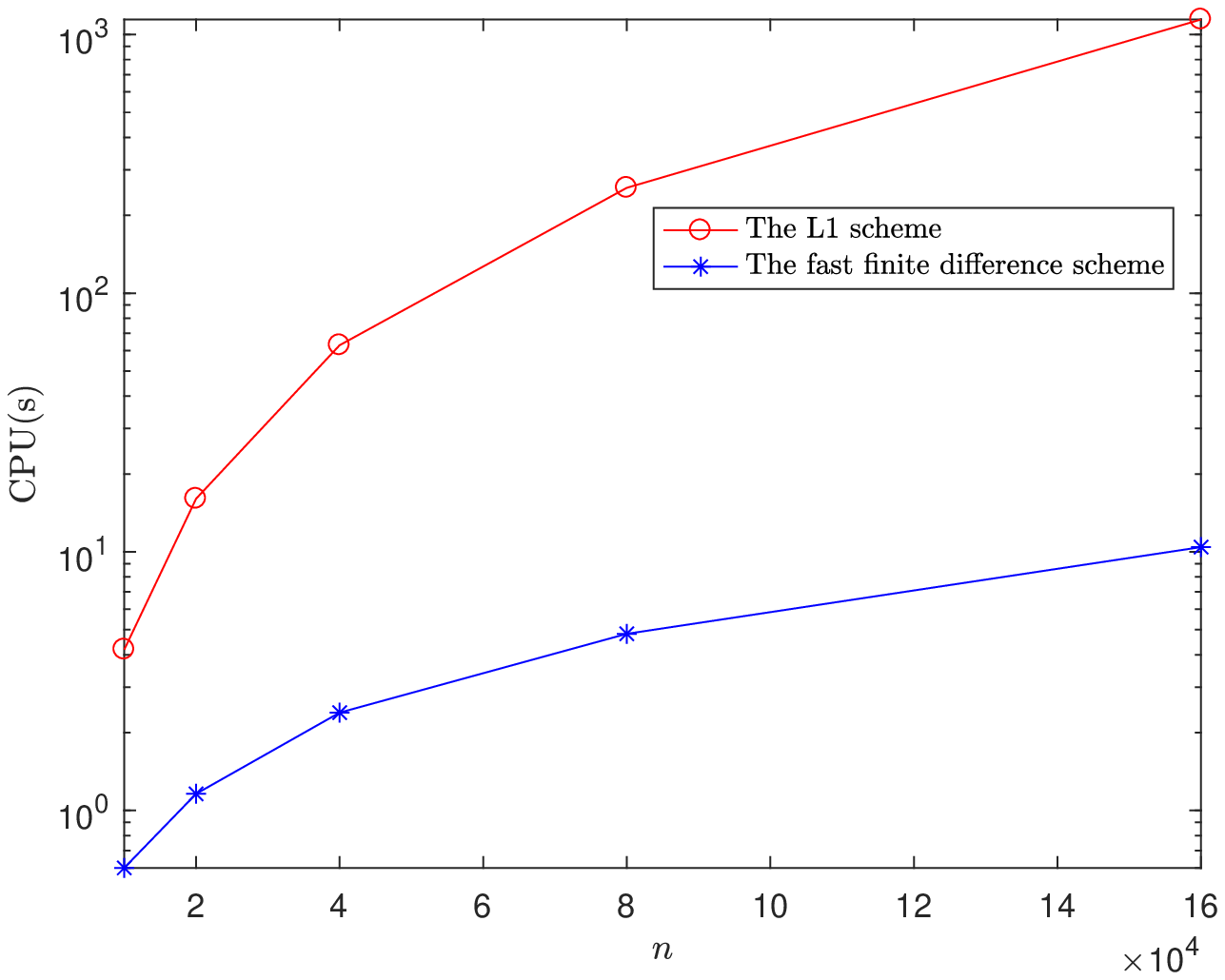}}  {\includegraphics[width=0.45\textwidth]{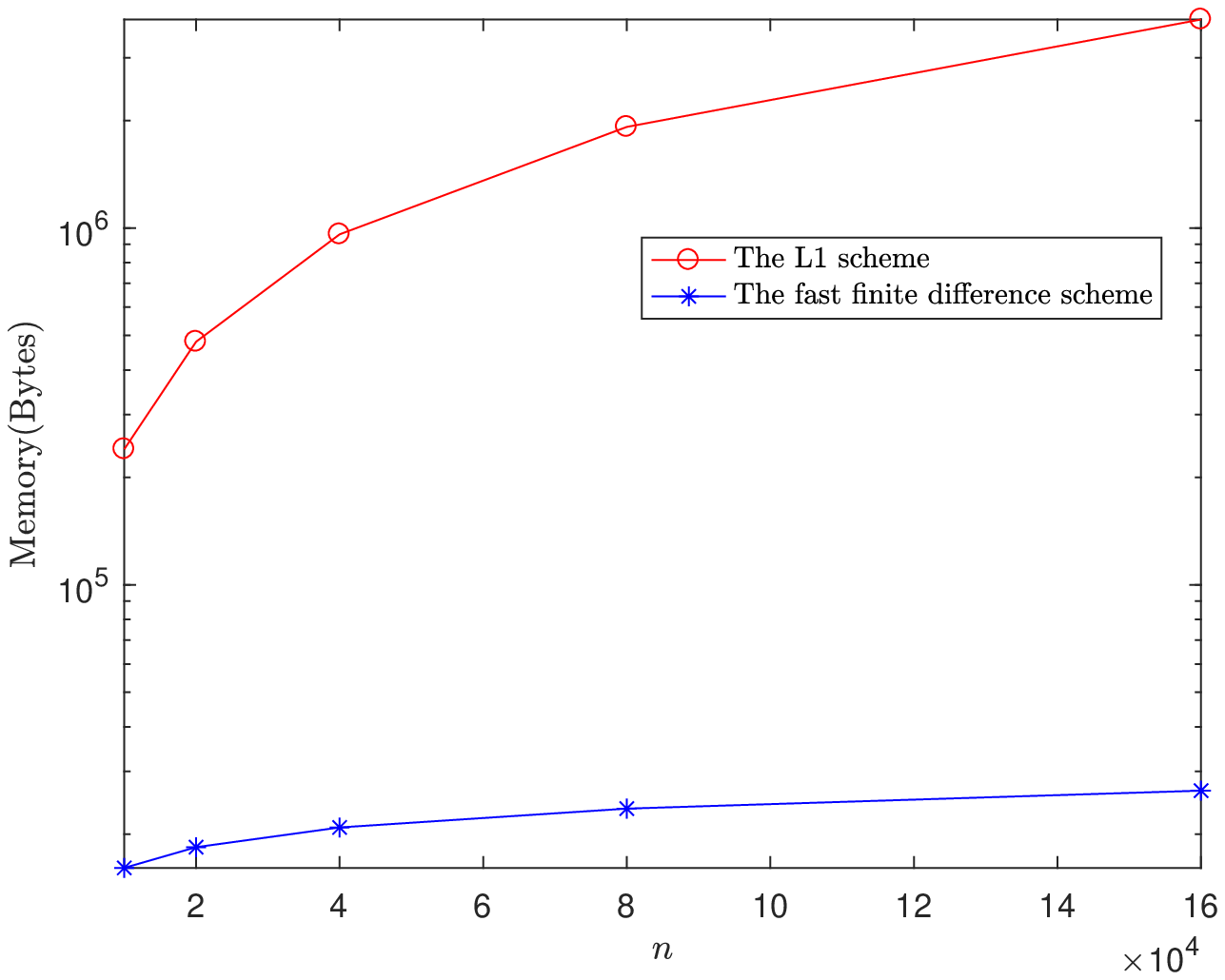}}\\
  \caption{The comparison of the $L1$ scheme and the fast ESA scheme $(\epsilon=\Delta t^2)$ on CPU(s) and Memory (Bytes) for Example \ref{example1} with $\alpha(t)=1-0.8t$.}\label{Fig1}
\end{figure}
The numerical results of the $L1$ scheme and the fast ESA scheme are listed in Table \ref{T1}, respectively. The fast ESA scheme achieves the same accuracy as the $L1$ scheme and reaches $2-\overline{\alpha}$ convergence order in time with the expected accuracy $\epsilon=\Delta t^2$, which verifies the correctness of the results in Section \ref{finite-difference-scheme}. Significantly, compared with the $L1$ scheme, the fast ESA scheme needs much less CPU time and memory with the same $n$ and $m$. Moreover, the number of exponentials in the fast algorithm needed is modest even for high accuracy approximations, which indeed contributes to reduce the storage and computational cost. In addition, Figure \ref{Fig1} shows the developments of CPU time and memory of the two schemes with respect to $n$. The CPU time and the memory of the $L1$ scheme increase much faster than those of the fast ESA scheme, which is consistent with the study in Section \ref{fast-approximation}.
\begin{example}\label{example2}
In this example, set $[0,x_R]=[0,1]$ and $[0,T]=[0,1]$. Consider the initial-boundary value problem of VO time-fractional diffusion equations (\ref{VOTFDE1})--(\ref{VOTFDE3}) with the source term
\begin{align*}
f(x,t)=20x^2(1-x)\left(\frac{t^{2-\alpha(t)}}{\Gamma(3-\alpha(t))}+\frac{t^{1-\alpha(t)}}{\Gamma(2-\alpha(t))}\right)-20(t+1)^2(1-3x),
\end{align*}
and the initial value
\begin{align*}
\varphi(x)=10x^2(1-x).
\end{align*}
The exact solution is given by (see \cite{Shen-2012})
\begin{align*}
u(x,t)=10x^2(1-x)(t+1)^2.
\end{align*}
\end{example}
\begin{table}[t]
\begin{center}
\caption{Convergence rates and the CPU time, memory of the $L1$ scheme and the fast ESA scheme for Example \ref{example2} with $m=1000$, $\epsilon=\Delta t^2$.}
\label{T2}
\def\temptablewidth{1.0\textwidth}
{\rule{\temptablewidth}{1.0pt}}
{\footnotesize
\begin{tabular*}{\temptablewidth}{@{\extracolsep{\fill}}cccccccccccccc}
& &\multicolumn{4}{c}{{$L1$\ \ scheme}} & \multicolumn{4}{c}{{Fast ESA scheme}} \\
 \cline{2-5}\cline{6-11}
$\alpha(t)$          &$n$  &$Err(\Delta x,\Delta t)$      &$Order_t$  &CPU(s)     &Memory   &$Err(\Delta x,\Delta t)$        &$Order_t$ &CPU(s)   &Memory  &$N_{\epsilon}$\\ \hline
$\frac{2+\sin(5t)}4$ &10000  &1.3566e-8   & -      &328.40     &8.02e+7  &1.6569e-8    &-     &39.53   &4.34e+6 &264 \\
                     &20000  &5.0973e-8   &1.41    &1356.85     &1.60e+8  &5.9808e-8    &1.47  &87.02   &4.95e+6 &302 \\
                     &40000  &1.9542e-9   &1.38    &5527.00     &3.20e+8  &2.1382e-9    &1.48  &159.06  &5.64e+6 &345 \\
                     &80000  &7.4773e-10   &1.39    &21462.24    &6.41e+8  &7.9534e-10    &1.43  &444.86  &6.35e+6 &389 \\
                     &160000 &3.0255e-10   &1.31    &87117.08    &1.28e+9  &3.1550e-10    &1.33  &990.03 &7.10e+6 &436 \\
\hline
$1-0.8t$             &10000  &3.1459e-8   & -      &341.48     &8.02e+7  &3.4922e-8    &-     &49.17   &5.45e+6 &333 \\
                     &20000  &1.4256e-8   &1.14    &1376.68     &1.60e+8  &1.5391e-8    &1.18  &105.49   &6.23e+6 &382 \\
                     &40000  &6.5467e-9   &1.12   &5342.77     &3.20e+8  &7.4250e-9    &1.05  &248.66  &7.07e+6 &434 \\
                     &80000  &3.0392e-9    &1.11    &21465.18    &6.41e+8  &3.2354e-9    &1.20  &553.78  &7.96e+6 &490 \\
                     &160000 &1.4491e-9   &1.07    &90444.47    &1.28e+9  &1.5451e-9    &1.07  &1226.20 &8.91e+6 &549 \\
\end{tabular*} {\rule{\temptablewidth}{1pt}}
}
\end{center}
\end{table}

With two different types of VO functions $\alpha(t)=\frac {2+\sin(5t)}4$ and $\alpha(t)=1-0.8t$, Table \ref{T2} lists the computational results of the $L1$ scheme and the fast ESA scheme with the expected accuracy $\epsilon=\Delta t^2$. Meanwhile, since only the performance in time is investigated, we take the fixed and sufficiently small spatial step size $\Delta x=1/1000$. The temporal step size $\Delta t$ verifies from $1/10000$ to $1/160000$. From Table \ref{T2}, we note that the two schemes have the same accuracy and convergence rate. However, the CPU time and the memory in workspace are both pretty less than the $L1$ scheme. In fact, the fast ESA scheme takes about 990s when $n=160000$, while the $L1$ scheme takes more than $8e+04$s. Moreover, the storage used in the fast ESA scheme is almost 1\textperthousand\ of the one in the $L1$ scheme when $n$ is large. The numerical results verify the effectiveness of the fast ESA scheme.

\section{Concluding Remarks} \label{concluding-remarks}
In this paper, an efficient fast algorithm for the VO time-fractional diffusion equations is presented applying the ESA technique with specified quadrature exponents. The parameters are properly selected to achieve the efficient accuracy. The computational cost of the proposed algorithm is of $\mathcal{O}(n\log^2 n)$ with $\mO(\log^2n)$ storage. Moreover, the resulting scheme is verified to be unconditionally stable and convergent with the order of $\mathcal{O}(\Delta t^{2-\overline{\alpha}}+\Delta x^2)$ via the maximum principle. Numerical tests show that the fast ESA scheme achieves the same accuracy and convergence order with much less storage and computational cost comparing with the $L1$ scheme.

In future work, the strategies developed in this paper could be exploited to construct fast algorithms for VO functions with respect to space variable as $\alpha(x;t)$ or in other definitions \cite{Sun-2019}. Furthermore, due to the nonuniform time step is a vital tool to approximate the fractional derivative, the fast method on the nonuniform time step is under our consideration.

\end{document}